\documentclass[10pt, conference]{ieeeconf}

\usepackage{graphics, graphicx, amsmath, float, amsfonts, amssymb, tikz, pgfplots, mathrsfs, mdframed, lipsum, accents, tabulary, booktabs}

\usepackage{aliascnt}
\usepackage{cleveref}

\newtheorem{thm}{Theorem}

\newaliascnt{cor}{thm}
\newaliascnt{lem}{thm}
\newaliascnt{prop}{thm}
\newaliascnt{con}{thm}
\newaliascnt{defn}{thm}
\newaliascnt{eg}{thm}
\newaliascnt{rem}{thm}

\newtheorem{cor}[cor]{Corollary}
\newtheorem{lem}[lem]{Lemma}
\newtheorem{prop}[prop]{Proposition}

\newtheorem{defn}[defn]{Definition}
\newtheorem{proofpart}{Part}

\newtheorem{rem}[rem]{Remark}

\aliascntresetthe{cor}
\aliascntresetthe{lem}
\aliascntresetthe{prop}
\aliascntresetthe{con}
\aliascntresetthe{eg}
\aliascntresetthe{rem}
\aliascntresetthe{defn}

\crefname{cor}{corollary}{corollaries}
\crefname{lem}{lemma}{lemmas}
\crefname{prop}{proposition}{propositions}
\crefname{con}{conjecture}{conjectures}
\crefname{defn}{definition}{definitions}
\crefname{eg}{example}{examples}
\crefname{rem}{remark}{remarks}

\newcommand{\bi}{\begin{itemize}}
\newcommand{\ei}{\end{itemize}}
\newcommand{\beq}{\begin{enumerate}}
\newcommand{\eeq}{\end{enumerate}}

\newcommand{\be}{\begin{equation}}
\newcommand{\ee}{\end{equation}}

\newcommand{\bc}{\begin{center}}
\newcommand{\ec}{\end{center}}
\newcommand{\bd}{\begin{defn}}
\newcommand{\ed}{\end{defn}}
\newcommand{\bt}{\begin{thm}}
\newcommand{\et}{\end{thm}}
\newcommand{\bp}{\begin{proof}}
\newcommand{\ep}{\end{proof}}
\renewcommand{\l}{\left}
\renewcommand{\r}{\right}

\DeclareMathOperator*{\argmax}{arg\,max}

\IEEEoverridecommandlockouts

\title{Optimality of Unconstrained Pulse Inputs to the Bergman Minimal Model}

\author{
	\authorblockN{Christopher Townsend and Maria M. Seron} 
	\authorblockA{Priority Research Centre for Complex Dynamic Systems and Control,\\ School of Electrical Engineering and Computing, University of Newcastle, Australia \\
	Emails: chris.townsend@newcastle.edu.au, maria.seron@newcastle.edu.au}
	}

\begin{document}

\maketitle

\begin{abstract}
	We characterise optimality of bolus insulin inputs, to the Bergman minimal model, by the predicted behaviour of the plasma glucose concentration for a given disturbance. The result is derived subject to the constraints that the plasma glucose concentration must attain but not go below a specified minimum value and the bolus input is rectangular. We give numerical examples of the results for the Hovorka model.
\end{abstract}

\section{Introduction}

Type one diabetes is a chronic disease affecting over thirty-eight million people \cite{wyou16}. Diabetics, typically, require the subcutaneuous administration of insulin to minimise plasma glucose concentrations whilst avoiding hypoglycaemia. Current treatment is invasive and often provides poor control.  Hence, much recent effort has been devoted to developing an artificial pancreas \cite{harv10} to automate treatment and better control plasma glucose concentrations.  


Understanding and modelling the dynamics of glucose regulation assists the development of such systems and further treatment improvements. A number of models of glucose regulation have been proposed (\cite{makr06, wili09, colm14}). Each is typically comprised of sub-systems describing different physiological processes such as insulin kinetics and glucose absorption.


Recently, research has focused on comprehensive models of glucose dynamics which are generally preferred to test treatment policies and control algorithms, for example \cite{man14}. Typically, these models are high order dynamic system with many parameters to ensure robustness to inter-individual variability. However, simpler models are useful to establish general theoretical properties that would otherwise be difficult to investigate analytically. Indeed, most models of glucose dynamics share certain analytic properties -- such as positivity the of the plasma glucose. Thus analytic results obtained for simpler models can give insights into the behaviour of more comprehensive models.


We focus here on the \emph{Bergman (Khandarian) Minimal Model} (\cite{berg05, good15, kand09}) which is a simplified model of glucose metabolism frequently used for virtual patient simulations and as the basis of more comprehensive models such as the Fabietti model (\cite{fabi06}) and the extensions of \cite{roy06} and \cite{roy07}.  The model \eqref{eq:eqs} is a non-linear continuous-time model comprising a set of first order linear ordinary differential equations which govern the subcutaneous, plasma and interstitial concentrations and effectiveness of insulin and a non-linear ordinary differential equation which governs the plasma glucose concentration $g(t)$:
\begin{align}
\label{eq:eqs}
\begin{split}
  \dot z &= -d z + dk u \\
  \dot y &= -c y + cz \\
  \dot x &= -ax + aby\\
  \dot g &= - hg + w
 \end{split}
\end{align}%
where all variables and constants are positive and $u(t)$ is the input function. The functions $h$ and $w$ in \eqref{eq:eqs} are:
\begin{align}
\label{eq:wandh}
\begin{split}
  h = x + G\\
  w = r + E
\end{split}
\end{align}%
where the function $r$ is a given bounded function. Specifically, the terms in \eqref{eq:eqs} and \eqref{eq:wandh} represent:

\begin{itemize}
 \item{$u(t), z(t), y(t)$ and $x(t)$ -- the delivery, subcutaneuos concentration, plasma concentration and insulin effectiveness, respectively. }
 \item{$c$ and $d$ -- inverse time constants.}
 \item{$a, b$ and $k$ -- the insulin motility \cite{roy07}, insulin sensitivity and the clearance rate.}
 \item{$g(t)$ -- the plasma glucose concentration.}
 \item{$E$ and $G$ -- the endogenous glucose production and the effect of glucose on the uptake of plasma glucose and suppression of endogenous glucose production.}
 \item{$r(t)$ -- the glucose absorption from meals.}
\end{itemize}

Physiological values for the above are derived from \cite{kand09} and given in Table 1 of \cite{good15}. 

We contribute to the theoretical understanding of this model by characterising the magnitude, delivery time and duration of insulin bolus inputs that are optimal in the sense that they give the lowest maximum glucose concentration whilst avoiding hypoglycaemia, see \Cref{defn:lambda,defn:gamma,defn:glob}. Specifically, we impose a fixed constraint on the minimum glucose concentration and focus on lowering the maximum glucose concentration.  We show that this fixed constraint induces a fundamental limitation on the controllability of the maximum glucose concentration when the control input is a pulse. We constrain the minimum glucose concentration because the risks associated with hypoglycaemia are, generally, far greater than those associated with hyperglycaemia. To ensure robustness against uncertainties this constraint could be set above the hypoglycaemic threshold.

The effect of a fixed constraint on the control of plasma glucose concentrations has been investigated in \cite{good15}. The authors consider a discretised non-linear model, derived from the Bergman model,  which they use to derive a non-linear insulin bolus dosing algorithm. However, the bolus is constrained to be an impulse applied contemporaneously with an impulsive food input. This is a specific example of the cases considered here. For instance, in \cite{good15}, the duration $\tau = 0$ is fixed and $w$ is assumed to be the response of a second order system to a single impulse.

In \cite{town17} a pulse input $u(t)$ of fixed duration was shown to be optimal, i.e. it minimised the global maximum glucose concentration, if and only if either the fixed minimum glucose concentration occured between two global maxima of $g(t)$ or the global maximum occured between two fixed minima of $g(t)$. Here, we present the counterpart of these results by giving conditions on inputs of varying durations but fixed delivery time to minimise the global maximum glucose concentration. Furthermore, our main contribution is to generalise the results to pulse inputs of any duration and delivery time. This fully characterises optimality of arbitrary pulse inputs in the sense of minimising the maximum glucose concentration subject to a fixed constraint on the minimum glucose concentration.

The observations of this work are that: firstly, distinct optimal inputs -- in the sense of \cite{town17}, see \Cref{defn:lambda,defn:gamma} --  must intersect at least twice if one has a lower global maximum glucose concentration and, secondly, that pulse inputs of varying duration can intersect at most twice and will only do so if one input is \emph{nested} inside the other. Our results confirm the intuition that responses with a maximum between two minima result from longer pulses than responses with a minimum between two maxima. Finally, decreasing the duration for the first type of response or increasing the duration the second type of response will lower the global maximum. As $g$ is a continuous function of the duration the lengthening and shortening of the duration converges.
   

	The presented results reveal a fundamental limit on the controllability of the plasma glucose concentration achievable from a bolus input to the Bergman minimal model and allow the optimality, of the input, to be determined from the shape of the glucose response. They also specify the effect of changes to the parameters of a bolus input on the maximum plasma glucose concentration. This may, for example, act as metric for the optimality of control algorithms designed for artificial pancreas systems and assist in the determination of bolus guidelines. Regardless of our focus on the Bergman model, other models may be analysed \emph{mutatis mutandis}, see Section~\ref{sec:oth}.

\subsection*{Notation:}

We adopt the following notation throughout: $\overline u$ and $ \hat u$  are  the basal input and the magnitude of the bolus input; $\lambda$ and $\gamma$  are the global minimum glucose concentration and the global maximum glucose concentration; $t', t_{i,\max}, t_{i,\min}$ and $\tau$ are  the delivery time, the $i^{\text{th}}$ time when the glucose concentration is at its global maximum, the $i^{\text{th}}$ time when the glucose concentration is at its minimum and the duration of the interval over which the bolus is delivered;  $u(t, \tau) = u(t,A)$  is  the input $u(t)$ applied over the interval $A := [t', t'+ \tau]$; $g(h(u),w)= g(t, \tau)$  is  the reponse of $g$ to the functions $h$ and $w$, where $h(u)$ is the response of $h$ to the input $u(t,\tau)$; $t_i$ and $t_{g,i}$  are intersection points of the responses $h(u)$ and $h(v)$ for distinct inputs $u \neq v$ and the $i^{\text{th}}$ intersection point of the resulting $g(h(u), w)$ and $g(h(v),w)$ and, lastly,  $\gamma(u)$  is the global maximum of $g(h(u), w)$.

\section{Assumptions and Preliminaries}
\label{sec:prelim}
Regardless of the nominal defintions given above, we do not require $r$ to be a positive bounded function corresponding to the glucose absorption from meals nor $E$ to be the endogenous glucose production. Instead, we require that $w$ is a positive function bounded below by any positive real $\overline E \geq G$. This allows, for example, $r$ to be negative if $\overline E < E$. By abuse of notation we denote $\overline E$ by $E$.

Throughout we impose the following initial conditions: $z(0) = y(0) = k u(0)$, $x(0) = bk u(0)$ and $g(0) > 0$. We assume the function $w$ is positive and bounded. We also assume the input $u$ is positive and bounded and  of the form:
\begin{equation}
	\label{eq:u}
	  u(t, A) = \bar{u}+ \hat{u} \chi_A (t)
\end{equation}
where the constant $\bar u$ is the \emph{basal} input, $\hat u$ is the magnitude of the \emph{bolus} input applied at some time $t'$, known as the \emph{delivery time} and $\chi_A$ is the characteristic function of $A$. The bolus input is held constant over $A=[t',t' + \tau]$, where $\tau \in \mathbb{R}_+$. When $\tau = 0$ we define $u(t) := \hat u \delta_{t'} (t)$, where $\delta(t)$ is the Kronecker delta. The boundedness and positivity of $u(t)$ imply that $h$, given by \eqref{eq:eqs} and \eqref{eq:wandh}, is a continuous, positive and bounded function. We desire that there exist $\lambda > 0$ such that $g(t) \geq \lambda$ for all $t$. This is achieved if $\lambda$ is a global minimum of $g(t)$. We denote by $t_{min} \in \mathbb{R}_+$ a point such that $g(t_{min}) = \lambda$.  


Finally, unless otherwise stated we assume that $t_{\max} := \argmax_t {g(t)} < \infty$. The maximal time $t_{\max}$ exists if $w$ is assumed to vanish to its lower bound at infinity. \Cref{thm:port} summarises a number of useful results from \cite{town17}.  

  \begin{defn}[Steady-State]
	  The \emph{steady-state} of $g$ is $g(\infty):=\lim_{t \to \infty} g(t)$, when $\lim_{t \to \infty}u(t)  = \overline u$ and $\lim_{t \to \infty} w(t) = E$ i.e. it is the limit of the response of $g(t)$ when the only input is the constant input $\overline u$.
  \end{defn}

\begin{thm}[Portmanteau]
	\label{thm:port}
Suppose $h$ and $w$ are bounded positive real-valued functionals, $g$ is as in \eqref{eq:eqs}, $u(t)$ is as in \eqref{eq:u} and choose $\lambda \leq g(0)$ and $\tau \geq 0$. Then:
\begin{enumerate}
	\item{Under the assumed initial conditions, $x(\overline u) = bk\overline u$ for all $t$. Furthermore, $\lim_{t \to \infty} x(t) = bk \overline u$.}
  
  \item{$g(t)$ is a strictly monotone function of $u(t,\tau)$.}

  \item{Setting:
  \begin{align}
	  \label{eq:steady}
	  \overline u = \frac{1}{kb}\l (\frac{E}{g(0)} - G \r)
  \end{align}%
  gives $g(\infty) = g(0)$.}%
  \end{enumerate}
  \end{thm}

%
\begin{defn}[Proper Input]
	For some $\lambda \leq g(0)$, an input function, $u(t, \tau)$, is \emph{proper}, if there exists $t_{\min}$ such that $g(h(u(t_{\min})),w) = \lambda$ and $g(t) \geq \lambda$ for all $t$.
\end{defn}  
  Theorem \ref{thm:bolus} \emph{(Theorem 7 \cite{town17})} proves the existence of a bolus input delivered at any $t'$ and $\tau$ which achieves a specified minimum $\lambda > 0$ and thus proves the existence of proper inputs of the form \eqref{eq:u}.   

\begin{thm}[Insulin Bolus, (Theorem 7  \cite{town17})]
\label{thm:bolus}
Suppose $u(t)$ is of the form \eqref{eq:u}. Fix $\tau$ and  $t'$ -- the input time i.e. $A:= [t', t'+ \tau]$, choose $\lambda \in (0, g(t')]$ and suppose $\bar u$ is as in \eqref{eq:steady}. Then there exists $\hat u$ such that $u(t)$ is proper.
\end{thm}

\section{Optimal Duration}

Two necessary and sufficient conditions were given in \cite{town17} for the response $g(t)$ to an input of the form \eqref{eq:u} with a fixed duration $\tau$ to be optimal.\footnote{As in \cite{town17} we say an input is \emph{optimal} if it results in the lowest maximum of $g(t)$ for all inputs of the same duration.} These conditions are summarised in Definitions \ref{defn:lambda} and \ref{defn:gamma}. 

\begin{defn}[$\lambda$--optimal]
	\label{defn:lambda}
	An optimal input $u$ is $\lambda$--optimal if the global maximum of $g(u)$ occurs between two minima.
\end{defn}

\begin{defn}[$\gamma$--optimal]
	\label{defn:gamma}
	An optimal input $u$ is $\gamma$--optimal if all minima of $g(u)$ occur between two global maxima.
\end{defn}

	For a fixed $w$, if the input $u$ is $\lambda$--optimal, respectively if $u$ is $\gamma$--optimal we say the response $g(u)$ is $\lambda$--optimal, respectively $\gamma$--optimal. We extend the results of \cite{town17} to inputs of the form \eqref{eq:u} which may have any duration $\tau$. In \Cref{defn:glob} we define global optimality of an input. In this section, firstly, we consider global optimality over the class of $\lambda$--optimal inputs and secondly global optimality over the class of $\gamma$--optimal inputs. Finally, we characterise global optimality over all proper inputs of the form \eqref{eq:u}. The following defines some useful notation.  

\begin{defn}
	Let $\gamma(u,w) := \max_t \{g(u(t),w(t))\}$. This is denoted $\gamma(u)$ for fixed $w$.
\end{defn} 

\begin{defn}[Globally Optimal]
	\label{defn:glob}
	An input $u$ is \emph{globally optimal} if $\gamma(u) < \gamma(v)$ for all $v \neq u$.
\end{defn}

Lemma \ref{lem:phis} specifies the maximum number of intersection points of responses to distinct inputs of the form \eqref{eq:u}.  

\begin{defn}[Nested Inputs]
	Two inputs $u$ and $v$ are \emph{nested} if $A \subset B$, where $A$ and $B$ are the intervals over which the boluses $\hat u$ and $\hat v$ are applied.
\end{defn}

Throughout we adopt the convention that for two inputs $u$ and $v$ times related to $u$ are denoted by $t$ and times related to $v$ are denoted by $s$.  

\begin{lem}
	\label{lem:phis}
	Suppose $u$ and $v$ are distinct inputs with delivery times $t'$ and $s'$ respectively. Then, for each solution $\phi \in \{x(t), y(t), z(t), h(t), g(t)\}$ to \eqref{eq:eqs}--\eqref{eq:wandh}, there are at most two $t_i > \min\{t', s'\}$ such that $\phi(u, t_i) = \phi(v,t_i)$ and these $t_i$ are distinct for all $\phi$ only if $u$ and $v$ are nested.
\end{lem}

\begin{proof}
	Observe, for $t > \min\{t', s'\}$ that $z(u,t) = z(v,t)$ only if $u-v$ changes sign. As $u-\overline u$ and $v-\overline u$ are rectangular $u-v$ can change sign at most twice. Implying that $z(u,t) = z(v,t)$ at most twice and thus $z(u,t) - z(v,t)$ may change sign at most twice. We proceed similarly for all solutions, $\phi$.
\end{proof}

\subsection{$\lambda$--Optimal Inputs}

\begin{lem}
	\label{lem:lambda}
	Suppose $g(t, \tau)$ and $g(t, \sigma)$ are the respective responses to the distinct $\lambda$--optimal inputs $u(t)$ and $v(t)$ with durations $\tau$ and $\sigma$, respectively. Then $\gamma(u) > \gamma(v)$ if and only if $[s', s' + \sigma] \subset [t', t' + \tau]$, where $t'$ and  $s'$ are the respective delivery times of the inputs $u(t)$ and $v(t)$.
\end{lem} 

\begin{proof}
	As $u$ and $v$ are pulse inputs there are at most two points $t_i$ at which the response $h(u)$ and $h(v)$ intersect, and similarly at most two $t_{g,i}$. We denote these $t_{i,1}$ and $t_{i,2}$, similarly $t_{g,1}$ and $t_{g,2}$. 
	
	Suppose $\gamma(u) > \gamma(v)$. Then $g(v) < g(u)$ for all $t$ in some interval $I \ni t_{1,\max}$. Additionally, $g(v) \geq g(u)$ at both $t_{1, \min}$ and $t_{2,\min}$. Thus $g(v) > g(u)$ for all $t \in (\min\{s', t'\}, t_{g,1})$, where $t_{g,1} \geq t_{1,\min}$. Otherwise there would exist more than two intersection points or $v$ would be non-proper. Therefore, $t' < s'$. Lastly, observe that should $s' + \sigma > t' + \tau$ then by \Cref{lem:phis} there is at most one $t_i > t'$ such that $h(u) = h(v)$ which implies that either $u$ or $v$ is not $\lambda$--optimal.


	Suppose instead that $[s', s' + \sigma] \subset [t', t' + \tau]$. As both $u$ and $v$ are $\lambda$--optimal, then by the above, $[s_{1,\min}, s_{2,\min}] \subseteq [t_{1,\min}, t_{2,\min}]$ or the converse. By assumption on the inputs we have that $g(u) < g(v)$ for all $t \in (t', t_{g,1})$. If $s_{1,\min} < t_{1,\min}$. Then we have that $t_{g,1} < t_{g,2} \leq t_{1,\min}$. Thus $g(v)$ would not be $\lambda$--optimal. As $g(v,t_{1,\max)} > \gamma(u)$ and $g(v) > \lambda$ for all $t > t_{g,2}$ which must occur before $t_{1,\min}$. Thus  $[s_{1,\min}, s_{2,\min}] \subseteq [t_{1,\min}, t_{2,\min}]$. Finally, as $g(v)$ and $g(u)$ intersect at most twice we have that $\gamma(u) > \gamma(v)$.
\end{proof} 
%


\begin{cor}
	Suppose $g(v)$ is $\lambda$--optimal for all $\tau > 0$. Then $u$ is globally optimal if and only if $A$ is a singleton.
\end{cor} 

\subsection{$\gamma$--Optimal Inputs}

\begin{lem}
	\label{lem:ys}
	Suppose $u$ and $v$ are distinct inputs for which there exists unique $t_i$ such that $h(u) = h(v)$ and $h(u) > h(v)$ for all $t \neq t_i$ and $t > \min\{ s', t'\}$. Then there are two distinct $p_i$ such that $y(u,p_i) = y(v,p_i)$, where $y(u, t)$ is the response of $y(t)$, from \eqref{eq:eqs} to the input $u$.
\end{lem} 

\begin{proof}
	Note that $t_i$ must be a minimum of the non-negative function $f(t) := x(u) - x(v)$. Thus there exists $\varepsilon > 0$ such that $x'(u) < x'(v)$ for $t \in (t_i - \varepsilon, t_i)$ and $x'(v) < x'(u)$ for $t \in (t_i , t_i + \varepsilon)$. By assumption $x(v) \leq x(u)$ for all $t$ thus, from \eqref{eq:eqs}, $f'(t)$ can only change sign about $t_i$ if $y(v) - y(u)$ changes sign about $t_i$. Hence, there is some $t_y < t_i$ such that $y(u) < y(v)$ and similarly there is some $s_y > t_i$ such that $y(v) < y(u)$. By continuity of $y(t)$ we see that there is a $p_2 \in (t_y, s_y)$ such that $y(u) = y(v)$. Additionally, $f'(t) > 0$ in some non-empty interval $[\min \{t' , s'\}, \min \{t' , s'\} + \delta)$ as $h(u) > h(v)$, for almost all $t>\min\{t', s'\}$. We have that $y(v) < y(u)$ on this interval. This implies that there must exist $p_1 < t_y$ such that $y(v) = y(u)$, again by continuity.
\end{proof} 

\begin{cor}
	\label{cor:ys}
	Suppose there are at most countably many $t_i$ such that $h(u) = h(v)$ and that $h(u) > h(v)$ for all $t \neq t_i$ and $t> \min\{t',s'\}$. Then for each $t_i$ there are two $p_i$ such that $y(u,p_i) = y(v,p_i)$.
\end{cor}

\begin{lem}
	\label{lem:gamma}
	Suppose $g(t, \tau)$ and $g(t, \sigma)$ are the respective responses to the $\gamma$--optimal inputs $u(t)$ with duration $\tau$ and $v(t)$ with duration $\sigma$. Then $\gamma(u) > \gamma(v)$ if and only if $[s', s' + \sigma] \supset [t', t' + \tau]$, where $t'$ and $s'$ are the delivery times of the inputs $u(t)$ and $v(t)$, respectively.
\end{lem}

\begin{proof}
	Assume that $\gamma(u) > \gamma(v)$ and suppose $s' > t'$. This implies that $t_{g,1} < t_{1,\max}$ as $g(u) < g(v)$ for all $t \in (t', t_{g,1})$. Hence $t_{g,2} \leq t_{\min}$. This would imply either that $g(v) \geq g(u) = \gamma(u)$ at $t_{2,\max}$ or by \Cref{lem:ys} that there are two additional $p_i$ such that $y(v,p_i) = y(u,p_i)$. This contradicts \Cref{lem:phis}. Hence $s' \leq t'$. Now suppose either $s' = t'$ or $s'+\sigma = t'+\tau$. This implies that there is at most one intersection point, and that $\hat v < \hat u$, which implies that either $g(u)$ or $g(v)$ is non-optimal. 

Suppose that $[s', s' + \sigma] \supset [t', t' + \tau]$. This implies that $g(u) > g(v)$ for all $t \in (t', t_{g,1})$. If $t_{g,1} \leq t_{1,\max}$. Then $t_{g,2} < s_{\min}$. Should $s_{\min} < t_{\min}$ then $g(v)$ would be non-optimal. Thus $s_{\min} \geq t_{\min}$ after which $g(v) \leq g(u)$ which implies $g(v)$ is not $\gamma$--optimal as $g(v, s_{1,\max}) \neq \max\{g(v, t) : t \geq s_{\min}\}$. Therefore $t_{g,1} > t_{1,\max}$. This together with the assumption that $g(v)$ is $\gamma$--optimal implies that $\gamma(v)  < \gamma(u)$.
\end{proof} 

\subsection{Amalgamation}

\begin{lem}
	Suppose $u$ is $\lambda$--optimal and $v$ is $\gamma$--optimal. Then $v$ is nested in $u$.
\end{lem} 

\begin{proof}
	For the sake of contradiction suppose $u$ is nested in $v$. We know that $t'$ must occur before $t_{1,\min}$ which, by the assumption that $u$ is nested in $v$ implies that $s'$ occurs before $t_{1,\min}$. Hence, there exists $t_{g,i}$ at or before each minimum of both $g(u)$ and $g(v)$. As there are at least three minima and at most two possible $t_{g,i}$ we see that $u$ cannot be nested in $v$. Instead, suppose, $u$ and $v$ are not nested. From \Cref{lem:phis} this implies that there is at most one intersection point of $g(u)$ and $g(v)$ contradicting optimality of $u$ or $v$.
\end{proof}

\begin{lem}
	Suppose $g(u)$ is $\lambda$--optimal and $g(v)$ is $\gamma$--optimal such that $\gamma(u) = \gamma(v)$. Then there exists an input $m$ with duration $\tau_{\lambda} > \sigma > \tau_{\gamma}$ such that $\gamma(m) < \gamma$.
\end{lem}

\begin{proof}
	Choose $\sigma$ as in the statement of the Theorem. We know the input $m$ is either $\gamma$--optimal or $\lambda$--optimal. In either case as it satisfies the conditions of Lemmas \ref{lem:lambda} and \ref{lem:gamma} we have that $\gamma(m) < \gamma$.
\end{proof} 

\begin{thm}
	\label{thm:mainmain}
	Suppose there exists $\tau > 0$ such that $g(v)$ is $\gamma$--optimal and $t$ such that $g(t) > g(0)$. Then an input  $u$ of the form \eqref{eq:u} is globally optimal if and only if the response has at least two global maxima interlaced between two minima.
\end{thm}

\begin{proof}
	Let $\tau$ be the duration as in the statement of the Theorem. By Lemma \ref{lem:gamma} so long as $\sigma > \tau$ produces a $\gamma$--optimal input $v$ then $\gamma(v) < \gamma(u)$. As $w$ vanishes at infinity, there exists a duration $\alpha$ such that the input $u$ is $\lambda$--optimal.

	Denote by $g(\alpha)$ the response of $g(t)$ to the input $u(t, \alpha)$ and by $g(\sigma)$ the response of $g(t)$ to the input $v(t, \sigma)$. We now construct a globally optimal $g$ and show that its shape is as in the statement of the Theorem. Recursively define the sequences $\overline \alpha := (\alpha_i)_{i=0} ^\infty$ and $\overline \sigma := (\sigma _i)_{i=0} ^\infty$ by $\alpha_0 :=  \alpha$ and $\sigma_0 := \sigma$ and $\alpha_i$ the least element of the following finite ordered partition of the interval $[\sigma_{i-1}, \alpha_{i-1}]$:

\begin{align*}
	{L}_i  := &\l \{\sigma_{i-1}, \frac{ (n-1) \sigma_{i-1} +  \alpha_{i-1}}{n}, \cdots \r . \\  &\quad \cdots,  \l. \frac{k_i \sigma_{i-1} + (n-k_i) \alpha_{i-1}}{n}, \cdots,  \alpha_{i-1} \r \}
 \end{align*}%
 where $n \in \mathbb{N}$ is arbitary and $k_i \leq n$, such that the response:
 \[ 
  g \l (\frac{k_i \sigma_{i-1} + (n-k_i) \alpha_{i-1}}{n} \r ) 
  \]%
 is $\lambda$--optimal. Similarly, $\sigma_i$ is defined to be the greatest element of $L_i$ such that:
 \[ 
    g \l (\frac{k_j \sigma_{i-1} + (n-k_j) \alpha_{i-1}}{n} \r )
 \]%
 is $\gamma$--optimal. The sequence $\overline \sigma$ is a monotone increasing sequence bounded above by $\alpha_i$ for all $\alpha_i \in \overline \alpha$ and therefore has a limit $\tau_-$. Similarly, $\overline \alpha$ is a monotone decreasing sequence bounded below by $\sigma_i$ for all $\sigma_i \in \overline \sigma$ and thus has a limit $\tau_+$. It remains to show that $\tau_- = \tau_+$. Suppose, for all $i$, that $\sigma_i < \alpha_i$. We see that if:
 \[ 
  \sigma _{i+1} = \frac{k_i \sigma_{i} + (n-k_{i}) \alpha_i}{n}
 \]%
 Then $\alpha_{i+1}$ must be the next element of $L_i$, as if were not the next element of $L_i$ would be $\gamma$--optimal contradicting our choice of $\sigma_{i+1}$, that is:
 \[ 
  \alpha _{i+1} = \frac{(k_i-1) \sigma_{i} + (n-k_{i}+1)\alpha_i}{n}
 \]%
 Thus:
 \begin{align*}
 \alpha_{i+1} - \sigma_{i+1} = \frac{1}{n} \l (\alpha_i - \sigma_i \r) =  
                              \ldots
                              = \frac{1}{n^{i+1}} \l (\alpha_0 - \sigma_0 \r)
 \end{align*}%
 
 i.e. $\lim_{i \to \infty} \l (\alpha_{i+1} - \sigma_{i+1} \r ) = 0$ i.e. $\tau_- = \tau_+$. Set $\tau := \tau_+$. 

Thus for all $\varepsilon \in (0, \tau)$ the optimal input with duration $\tau + \varepsilon$ must be $\lambda$--optimal an the optimal input with duration $\tau - \varepsilon$ must be $\gamma$--optimal. By continuity of $g$ there must be at least two equal maxima and two minima. 
	
As the limits, $\tau_-$ and $\tau_+$, are equal and $g$ is a continuous function of the duration we may consider the sequence $(g(\alpha_i))_{i=0} ^\infty$ to determine the shape of $g(\tau)$. Since the durations $\alpha_i$ decrease we have, as in the proof of \Cref{lem:lambda} that $[t_{1,\min, i+1}, t_{2, \min, i+1}] \subseteq [t_{1,\min, i}, t_{2, \min, i}]$ for each $i$. Thus:
	\begin{align}
		\label{eq:bestlabelthis}
		g(u_{i+1}) |_B \geq g(u_i)|_B
	\end{align}%
	where $B :=   [t_{1,\min, i}, t_{2,\min, i}]^c$ -- the complement in $\mathbb{R}_+$. Indeed: $\gamma(u_{i+1}) |_B \geq \gamma(u_i)|_B$. Therefore the response $g(u_{i+1})$ has no additional minima outside the interval $[t_{1,\min, i}, t_{2, \min, i}]$. By \eqref{eq:bestlabelthis}, the function $G(i) := \gamma(u_i) - \gamma(u_i) |_B$ is monotone decreasing, as the global maximum, $\gamma(u_i)$, is decreasing by \Cref{lem:lambda}, and bounded below by $0$. Suppose $\lim G(i) > 0$. This is true only if $\tilde u := \lim u_{i}$ is $\lambda$--optimal. If $g(\tilde u)$ is $\lambda$--optimal then there exists strictly positive $\varepsilon < \gamma(\tilde u) - \max \{ g(\tilde u): g(\tilde u) \neq \gamma(\tilde u) \wedge \dot g(\tilde u) = 0\}$. For all such $\varepsilon$ there is $\delta > 0$ such that $u(\tau - \delta)$ is $\lambda$--optimal as $g(\tau)$ is continuous. By \Cref{lem:phis} $\gamma(u(\tau - \delta)) < \gamma(\tilde u)$. This implies that $\tilde u$ is not the limit of the sequence of $\lambda$--optimal inputs $u_i$ with durations $\alpha_i$. Lastly, $\lim G(i) = 0$ implies that $\gamma(\hat u ) = \gamma( \hat u)|_B$ i.e. the maxima outside the interval $[t_{1,\min}, t_{2,\min}]$ are equal to the maxima inside the interval.
	
	Suppose $g(u)$ is as in the Theorem but there exists proper $v \neq u$ such that $\gamma(v) < \gamma(u)$. Thus, for $j \in \{1,2\}$, we require that $g(v) < g(u)$ at $t_{j,\max}$ and $g(v) \geq g(u)$ at $t_{j,\min}$. Therefore there must be more than two $t_i$ such that $h(u)=h(v)$, unless both $t_{g,i}$ occur at $t_{j,\min}$. In this case, by \ref{cor:ys}, there must be at least four distinct points at which $y(u) = y(v)$ contradicting \Cref{lem:phis}. 
\end{proof}

It is not possible to lower the global maximum of $g$ if $g(t) \leq g(0)$ for all $t$. However, in this case the shape of the reponse, $g$, specified in \Cref{thm:mainmain} does minimise the maximum of $h$. This is shown by \Cref{prop:prop}.  

\begin{prop}
	\label{prop:prop}
	Suppose $g(v) \leq g(0)$ for all $t$ and proper inputs $v$. Additionally, suppose, there exists $u$ for which there are two minima and there is a $t \in (t_{1,\min}, t_{2,\min})$ such that the response $g(u(t)) = g(0)$. Then $\max\{h(u)\} < \max\{h(v)\}$	for all proper $v \neq u$.
\end{prop}

\begin{proof}
	The proof follows if all such $v$ are nested in $u$ as $h$ is a monotonic function of $\hat u$. As $g(u)$ and $g(v)$ may intersect at most twice we have that $t_{g,1} \in [t_{1,\min}, t_{1,\max}]$ and $t_{g,2} \in [t_{1,\max}, t_{2,\min}]$ and that $g(v) < g(u)$ for all $t \in (t_{g,1}, t_{g,2})$. Thus $g(v) > g(u)$ for all $t > t'$ such that $t \not \in [t_{g,1}, t_{g,2}]$. This occurs only if $v$ is nested in $u$.
\end{proof}

\subsection{Optimal Duration for a Fixed Delivery Time}

We conclude this Section by characterising the optimality of an input with varying duration and fixed delivery time. \Cref{thm:fixed} is the analogous result for durations to the results for delivery times derived in \cite{town17}. Part 1 of \Cref{thm:fixed} is a generalisation of the main result, \emph{(Theorem 16)}, of \cite{good??}, when restricted to the class of rectangular inputs. This generalisation stems from only assuming that $w$ is bounded and vanishes at infinity and as we do not require that the global maximum, of $g$, occurs before its global minimum\footnote{This assumption holds if case \emph{A} from \Cref{thm:fixed} holds for $\tau = 0$. As, if case \emph{A} holds for $\tau = 0$. Then it holds for all $\tau \geq 0$.} for all inputs $u$, of the form \eqref{eq:u}. The results of \cite{good??} hold for more general inputs of the form $u(t) = \overline u  + \hat u (t',t)$ where $\hat u(t',t)$ is a positive bounded function such that $\hat u(t',t) = 0$ for all $t < t'$ and $\overline u$ is as in \eqref{eq:u}.

\begin{thm}
	\label{thm:fixed}
	Consider the following two cases for $g(t)$:
\begin{enumerate}
	\item[\emph{A.}] no global maximum occurs after a global minimum. 
	\item[\emph{B.}] no global maximum occurs before a global minimum. 
\end{enumerate}

Fix $t'$. Let $u(t', \tau)$ and $v(t', \sigma)$ be two distinct inputs delivered at $t'$. Suppose either: $u$ and $v$ satisfy \emph{A}, $u$ and $v$ satisfy \emph{B.} or $u$ satisfies \emph{A} and $v$ satisfies \emph{B}. Then, for each respective case:

\begin{enumerate}
	\item{$\gamma(u) < \gamma(v)$ if and only if $\tau < \sigma$.}

	\item{$\gamma(u) < \gamma(v)$ if and only if $\tau > \sigma$.}
	
	\item{there exists $\alpha \in (\sigma, \tau)$ such that $\gamma(m(t',\alpha)) < \min\{\gamma(v), \gamma(u)\}$. Furthermore, for all $\alpha \not \in (\sigma, \tau)$ the maximum $\gamma(m(t',\alpha)) \geq \min\{\gamma(v), \gamma(u)\}$.}

\end{enumerate}
\end{thm}

\begin{proof}
	Throughout this proof we say $g(t, \tau)< g(t, \sigma)$ \emph{initially} if there exists $\varepsilon > 0$ such that $g(t, \tau)< g(t, \sigma)$ for all $t \in (t', t'+ \varepsilon)$. 	

	\begin{proofpart}
		Suppose $\tau < \sigma$. Then, by \Cref{thm:port} Part 2 and as $u$ and $v$ are proper, $\hat u > \hat v$. Thus, initially $g(t, \tau) < g(t, \sigma)$. By \Cref{lem:phis} there is at most one $t_{g}$ at which $g(t_g, \tau) = g(t_g, \sigma)$. As $u$ and $v$ are proper this $t_g$ must exist and $t_g \leq s_{1,\min}$. Otherwise $g(s_{1,\min}, \tau) < \lambda$. Note that \Cref{lem:phis} implies that $g(u) - g(v)$ must change sign at $t_g$ even if $t_g =  s_{1,\min}$. If $t_g \in (\max\{s_{\max}\}, s_{1,\min}]$, where $\max\{s_{max}\}$ is the greatest time at which $g(v)$ is maximised. Then $\gamma(u) < \gamma(v)$ and $\{t : g(t,\tau) =\lambda\} \subset (\max\{s_{\max}\}, s_{1,\min}]$ i.e. all minima of $g(u)$ occur between the last maximum of $g(v)$ and the first minimum of $g(v)$. 
		
		Instead, suppose $t_g \leq \max\{s_{\max}\}$. This implies that $\gamma(u)$ occurs after $t_g$. Thus, as $g(t, \tau) > \lambda$ for all $t \leq \max\{t_{\max}\}$ and $g(u) > g(v)$ for all $t> t_g$, $u$ is not proper.

		Suppose $\gamma(u) < \gamma(v)$. If initially $g(t, \tau) < g(t,\sigma)$ then $\tau < \sigma$. Suppose, initially $g(t, \tau) > g(t,\sigma)$. As $g(v) > g(u)$ for all $t > t_g$ we have that $\max\{s_{\min}\} \leq t_g$. Otherwise $u$ would not be proper. By assumption $\max\{s_{\max}\} < \min\{s_{\min}\}$. Therefore $\max\{s_{\max}\} < t_g$. Then as there is at most one intersection point, of $g(u)$ and $g(v)$, we see that $t_g > \max\{t_{\max}\}$. Contradicting our assumption that $\gamma(u) < \gamma(v)$.
	\end{proofpart}

	\begin{proofpart}
		Suppose $\tau > \sigma$. This implies that initially $g(t, \tau)> g(t,\sigma)$. Hence, similarly to Part 1 above, we see that $t_g \geq \max\{s_{\min}\}$. Thus all minima of $g(v)$ must occur before $t_g$ and all minima of $g(u)$ must occur after $t_g$. Thus $t_g \in [\max\{s_{\min}\}, \min\{t_{\min}\}]$. In particular, this implies that $t_g < \min\{t_{\max}\}$. As $g(v) > g(u)$ for all $t > t_g$ we have, by the assumed shape of $g(u)$, that $g(t_{\max}, \sigma) > \gamma(u)$ i.e. $\gamma(v) > \gamma(u)$.
		

		Suppose $\gamma(u) < \gamma(v)$. If $t_g > \min\{t_{\min}\}$. Then $g(v) > g(u) \geq \lambda$ for all $t < t_g$. Otherwise there would exist $s$ such that $g(v) < \lambda$. Additionally, $g(v) < g(u) \leq \gamma(u)$ for all $t > t_g$. Thus all $s_{\max} < t_g$ and as no minimum of $g(v)$ exists after $\min\{s_{\max}\}$, we see that $v$ is not proper. Hence, $t_g \leq \min\{t_{\min}\}$. If  $g(v) > g(u)$ initially. Then $\max\{g(v,s)\} < \max\{g(v,t)\} <  \gamma(u)$ for $s \in [t', t_g]$ and $t > t_g$. This is because no global maximum may occur before the last minimum of $g(v)$. Hence $g(v) < g(u)$ initially. This implies that $\sigma < \tau$.
	\end{proofpart} 

	\begin{proofpart}
		We note that the problem is well-posed as if $\sigma > \tau$ for $u$ satisfying \emph{A} and $v$ satisfying \emph{B}. Then initially $g(v)>g(u)$. Thus $t_g \leq \min\{s_{\min}\}$, after which $g(u) > g(v) \geq \lambda$. This implies either $u$ is not proper or $t_{g} = \min\{s_{\min}\}$. In which case there exists $s < t_g$ such that $g(v) > \gamma(u)$. As $v$ satisfies \emph{B} there must exist $s > t_g$ such that $g(v) > \gamma(u) > g(u)$. Contradicting the uniqueness of $t_g$.


		Suppose $\alpha < \sigma$. Then $m(t', \alpha)$ satisfies \emph{B} as $g(m) < g(v)$ initially. Thus, by the above, $\gamma(m) > \gamma(v)$. Similarly for $\tau$.
		
		As $g$ is a continuous function of the duration there exists $\alpha < \tau$ such that $m(t', \alpha)$ satisfies \emph{A} and by the above $\gamma(m) < \gamma(u)$ and $\alpha > \sigma$. Similarly for $v$.


		\end{proofpart}
		\end{proof} 

\Cref{cor:tried} characterises when it is better to optimise the delivery time instead of the duration of an input.  

\begin{cor}
	\label{cor:tried}
	Fix $t'$. There exists delivery time $s' \neq t'$ such that $\gamma(s', \tau) < \gamma(t', \tau)$ for all $\tau$ if and only if $u(t', \tau)$ satisfies either \emph{A} or \emph{B} of \Cref{thm:port} for all $\tau$.
\end{cor}

We conclude this section with \Cref{cor:forpart3} which extends \Cref{lem:lambda,lem:gamma} to the case of a fixed input time.  

\begin{cor}
	\label{cor:forpart3}
	Fix $t'$. Suppose $u(t', \tau)$ is proper and  either $\lambda$ or $\gamma$--optimal. Then $\gamma(v(t', \sigma)) > \gamma(u(t', \tau))$ for any $\sigma \neq \tau$ where $v(t', \sigma)$ is proper. Futhermore, such $v$ is neither $\lambda$ nor $\gamma$--optimal.
\end{cor} 

\begin{proof}
	This follows by similar argument to the proofs of \Cref{lem:lambda,lem:gamma}.	
\end{proof}

\section{Algorithm for Optimal Duration}

As the duration, $\tau$, is bounded below by $0$ the following algorithm may be used to locate the optimal duration: 

\subsection*{Algorithm:}
\begin{enumerate}
	\item{Set $\tau = 0$ i.e. $\chi_A  = \delta$. If the response $g$ is $\lambda$--optimal then $\tau$ is globally optimal. Otherwise:}
	\item{Choose $\tau > 0$:}
		\begin{enumerate}
			\item{if $g$ is $\lambda$--optimal then proceed to step 3}
			\item{otherwise increase $\tau$ until $g$ is $\lambda$--optimal}
		\end{enumerate}
	\item{Recursively bifurcate the interval $[\sigma, \alpha]$, where $\sigma$ is the largest known $\tau$ such that $g$ is $\gamma$--optimal and $\alpha$ is the least known $\tau$ such that $g$ is $\lambda$--optimal.}
\end{enumerate} 

\begin{rem}
If the condition that $g$ is $\lambda$--optimal is replaced by condition \emph{B} from \Cref{thm:fixed} and the condition that $g$ is $\gamma$--optimal is replaced by condition \emph{A} from \Cref{thm:fixed}, this algorithm may be adapted to find the optimal duration for a fixed delivery time $t'$. 
\end{rem}

\subsection*{Numerical Example:}
In the example presented in Figures \ref{fig:gprofs}--\ref{fig:durs}, the algorithm to locate the optimal duration was applied to a system where the parameters of \eqref{eq:eqs} and \eqref{eq:wandh} were chosen to be:  $d=0.025$, $k=1806^{-1}$, $c=0.025$, $a=0.0101$, $b=8.16 \times 10^{-4}$, $G=0.0023$, $E=1.0$, and $r(t) = 263^{-1} f_1(t)$, where $f_1 (t)$ is the solution to the system of linear differential equations:
\[
	\begin{pmatrix}
		\dot f_1(t) \\
		\dot f_2(t)
	\end{pmatrix} 
	= 
	\l ( \frac{1}{60} \r)
	\begin{pmatrix}
		-1 &  1 \\
		0  & -1
	\end{pmatrix}
	\begin{pmatrix}
		f_1(t) \\
		f_2(t)
	\end{pmatrix} 
	+
	\begin{pmatrix}
		0 \\
		\rho(t)
	\end{pmatrix}
\]%
where: $\rho(t) := 5 \chi_{[300, 800]}(t) + 100\chi_{[450,460]}(t)$. We take the initial conditions to be as in Section \ref{sec:prelim} and set $g(\infty) = g(0) = 5.0 \mathrm{mmolL}^{-1} \, (90 \mathrm{mgdl}^{-1})$. The minimum glucose concentration $\lambda$ is chosen to be $4.0 \mathrm{mmolL}^{-1} \, (72 \mathrm{mgdl}^{-1})$. 
 
For computational reasons the smallest duration tested was $\sigma = 2$. The longest duration considered was $\alpha = 1000$. In Figures \ref{fig:gprofs} and \ref{fig:durs} the blue, green, dashed black, red and cyan lines correspond to the durations $\tau = 100, 250, 370, 550$ and $600$ respectively.

 \begin{figure}[]
	 \begin{flushleft}
		 \includegraphics[width=8cm, height=9.7cm, keepaspectratio]{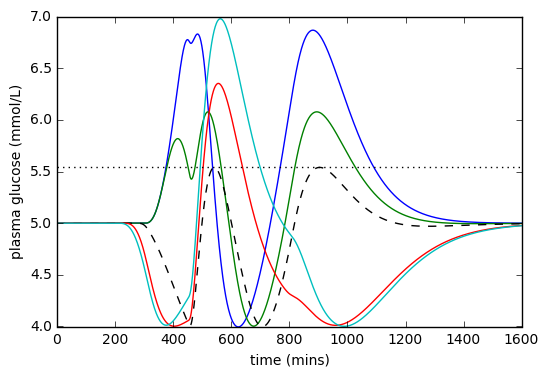}
	\end{flushleft}
	\caption{Selected plasma glucose concentration $g(t)$  profiles for various input durations. The dashed black profile meets the global optimality conditions.}
	 \label{fig:gprofs}
 \end{figure}

 As the duration approaches $\tau = 370$, which corresponds to the dashed black profile of Figure \ref{fig:gprofs}, the maximum glucose concentration decreases. This is shown in Figure \ref{fig:maxgs} Indeed $\gamma(u)$ is monotonic as $\tau \to^- 370$ and monotonically decreasing as $\tau \to^+ 370$. The small deviations are an artifact of the numerical precision. With no optimisation: $\gamma(u(100)) = 7.2 \mathrm{mmolL}^{-1} \, (129 \mathrm{mgdl}^{-1})$ and $\gamma(u(700)) = 8.3 \mathrm{mmolL}^{-1} \, (149 \mathrm{mgdl}^{-1})$. Whilst $\gamma(u(370)) = 5.54 \mathrm{mmolL}^{-1} \, (100 \mathrm{mgdl}^{-1})$.

 \begin{figure}[H]
	 \begin{flushleft}
		 \includegraphics[width=8cm, height=9.7cm, keepaspectratio]{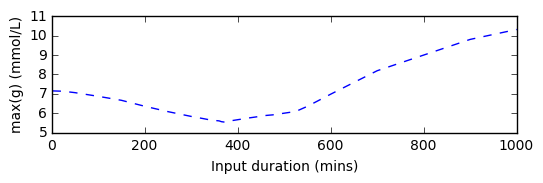}
	\end{flushleft}
	\caption{Maximum plasma glucose concentration, $\gamma(u)$, as a function of the input duration $\tau$.}
	 \label{fig:maxgs}
 \end{figure}

 Lastly, Figure \ref{fig:durs} shows $\hat u \chi_A$ for the $A$ yielding glucose profiles shown in Figure \ref{fig:gprofs}. Each interval over which the input $u(t, \tau) \neq \overline u$ is nested in the next larger interval.

\begin{figure}[H]
	 \begin{flushleft}
		 \includegraphics[width=8cm, height=9.7cm, keepaspectratio]{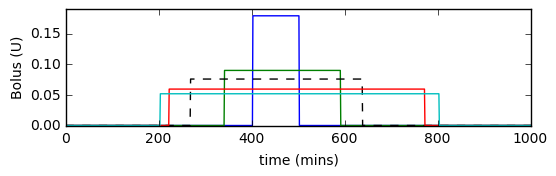}
	\end{flushleft}
	\caption{The functions $u(t,\tau) - \overline u$ for selected input durations}
	 \label{fig:durs}
 \end{figure}

 \begin{rem}
	 Figure \ref{fig:maxgs} indicates that the rate of decrease in $\gamma(u)$ drops about the optimal duration i.e. $\frac{d}{d\tau} \gamma(u,\tau) \to 0$.  Therefore, it seems that there is little benefit in over-optimising the duration.
 \end{rem}

\section{Application to Other Models}
\label{sec:oth}

For models of the form:
\begin{align}
	g = f(g, u, h, w, t)
\end{align}
where $g$ is the plasma glucose concentration, $u$ is the insulin input and  $w$ and $h$ are some bounded positive functions, the results of \cite{town17} and those presented here require that:
\begin{enumerate}
	\item{$g$ is a continuous function of $u, w$ and $t$ that decreases monotonically with respect to $u$}
	\item{$h$ is continuous function of $u$ and $t$ and is monotone in $u$.}
	\item{$w$ and $h$ decay to their respective lower bounds.} 
	\item{$g \geq g(\infty) \geq \lambda$ if $u(t) = 0$, for all $t$, where $g(\infty)$ is the desired steady-state glucose concentration.}
\end{enumerate}

The Hovorka model (\cite{hovo04}) is another dynamic model of glucose metabolism which explicitly includes a number of physiological factors, for example a renal excretion term. We give numerical examples of our results for the Hovorka model which suggest it may satisfy our assumptions. The Hovorka model is:

\begin{align}
	\begin{split}
	\dot z & = -dz + u \\
	\dot y & = -dy + dz\\
	\dot x & = -kx + cdy\\
	\dot x_1 & = -a_1 x_1 + a_1 b_1 x \\
	\dot x_2 & = -a_2 x_2 + a_2 b_2 x \\
	\dot x_3 & = -a_3 x_3 + a_3 b_3 x \\
	\dot q_1 & = -h_1 q_1 + lq_2  + w \\
	\dot q_2 & = -h_2 q_2 + x_1 q_1
	\end{split}
\end{align}%
where:
\begin{align}
	\label{eq:hovw}
	\begin{split}
	w & = E + r\\
		h_1 & = V^{-1} \l ( f_c + f_r + \frac{x_3 E}{g} \r ) + x_1\\
	h_2 & = l + x_2
	\end{split}
\end{align}%
and:
\begin{align}
	\begin{split}
		f_r & := \begin{cases} VR \l (1-\frac{\overline g_r}{g} \r ), &g \geq \overline g_r\\
					  0			    , &\text{otherwise}
	\end{cases}\\
	f_c &:= \begin{cases} {f}{g^{-1}}, &g \geq \overline g_c\\
		  	      {f}{\overline{g}_c ^{-1}}, &\text{otherwise}
	\end{cases}
	\end{split}
\end{align}%
where $a_i, b_i, c, d, E, f, \overline g_c, \overline g_r, k, l, V$ and $R$ are positive constants, physiological values for which may be found in \cite{hovo04} and $r$ is a positive bounded function. The plasma glucose $g$ is a scalar mulitple of $q_1$:
\begin{align}
	\begin{split}
		g = V^{-1} q_1
	\end{split}
\end{align}%



As for the Bergman model, we assume $u$ is a positive bounded function of the form \eqref{eq:u} such that $\lim_{t \to \infty} u(t) = \overline u$. We also assume that $\lim_{t \to \infty} r(t) = 0$. The steady-state value has a positive upper bound, $g > 0$, for $\overline u = 0$. As $g$ is a continuous function of $u$, we may choose $\lambda$ and the steady-state value, $g(\infty) = g(0)$, to be any value less than this upper bound.

\subsection{Numerical Examples}

In the example presented in Figures \ref{fig:hovgprofs}--\ref{fig:hovdurs} the algorithm to locate the optimal duration was applied to the Hovorka model with parameters as in \cite{hovo04} and a body weight, on which the values in \cite{hovo04} depend, of $70 \mathrm{kg}$. The function $r(t) := \l( \frac{1}{55} \r ) f_1(t)$ where $f_1(t)$ is the solution to the differential equations:
\[
	\begin{pmatrix}
		\dot f_1(t) \\
		\dot f_2(t)
	\end{pmatrix} 
	= 
	\l ( \frac{1}{55} \r)
	\begin{pmatrix}
		-1 &  1 \\
		0  & -1
	\end{pmatrix}
	\begin{pmatrix}
		f_1(t) \\
		f_2(t)
	\end{pmatrix} 
	+
	\begin{pmatrix}
		0 \\
		0.8 \rho(t)
	\end{pmatrix}
\]%
where: $\rho(t) := 0.2 \chi_{[300, 800]}(t) + 5\chi_{[450,460]}(t)$. The initial conditions where set such that the system was in steady-state at $t = 0$ for a steady-state value $g(\infty) = g(0) = 5.0 \mathrm{mmolL}^{-1} \, (90 \mathrm{mgdl}^{-1})$. The minimum glucose concentration $\lambda$ was chosen to be $4.0 \mathrm{mmolL}^{-1} \, (72 \mathrm{mgdl}^{-1})$. 
 
The duration tested ranged from $\sigma = 2$ to $\alpha = 1000$. In Figures \ref{fig:hovgprofs} and \ref{fig:hovdurs} the blue, green, dashed black, red and cyan lines correspond to the durations $\tau = 150, 250, 305, 450$ and $550$ respectively.

 \begin{figure}[]
	 \begin{flushleft}
		 \includegraphics[width=8cm, height=9.7cm, keepaspectratio]{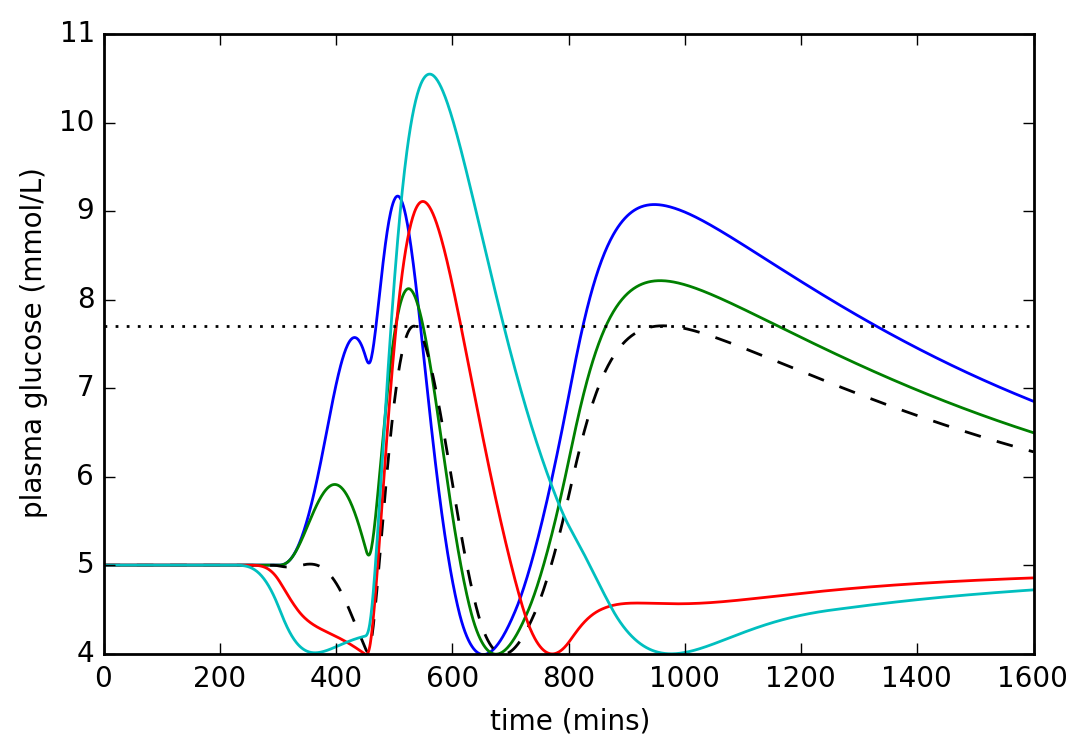}
	\end{flushleft}
	\caption{Selected plasma glucose concentration $g(t)$  profiles for various input durations. The dashed black profile meets the global optimality conditions, for the Hovorka model.}
	 \label{fig:hovgprofs}
 \end{figure}

 As the duration approaches $\tau = 305$, which corresponds to the dashed black profile of Figure \ref{fig:hovgprofs}, the maximum glucose concentration decreases monotonically from above and below. With no optimisation: $\gamma(u(2)) = 9.8 \mathrm{mmolL}^{-1} \, (176 \mathrm{mgdl}^{-1})$ and $\gamma(u(1000)) = 17.5 \mathrm{mmolL}^{-1} \, (315 \mathrm{mgdl}^{-1})$. Whilst, the optimal duration $\gamma(u(305)) = 7.7 \mathrm{mmolL}^{-1} \, (139 \mathrm{mgdl}^{-1})$.
 
 \begin{figure}[H]
	 \begin{flushleft}
		 \includegraphics[width=8cm, height=9.7cm, keepaspectratio]{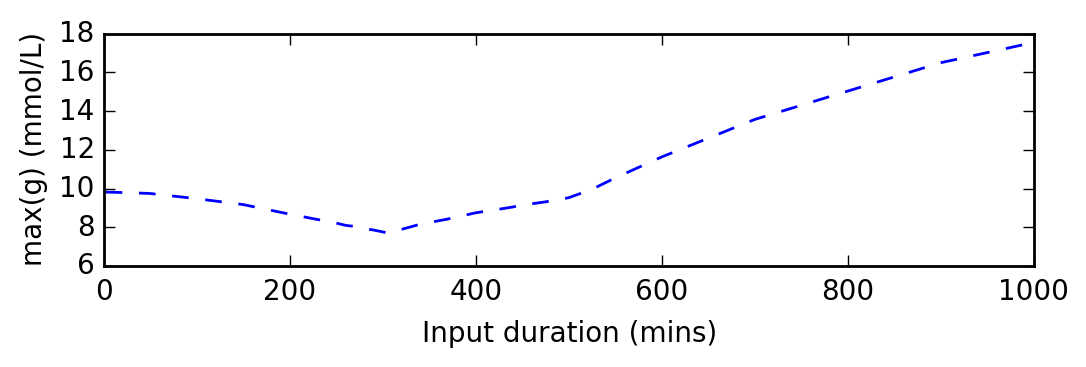}
	\end{flushleft}
	\caption{Maximum plasma glucose concentration, $\gamma(u)$, as a function of the input duration $\tau$, for the Hovorka model.}
	 \label{fig:hovmaxgs}
 \end{figure}

 Figure \ref{fig:hovdurs} shows $\hat u \chi_A$ for the $A$ yielding the glucose profiles shown in Figure \ref{fig:hovgprofs}. Each interval over which the input $u(t, \tau) \neq \overline u$ is nested in the next larger interval.

\begin{figure}[H]
	 \begin{flushleft}
		 \includegraphics[width=8cm, height=9.7cm, keepaspectratio]{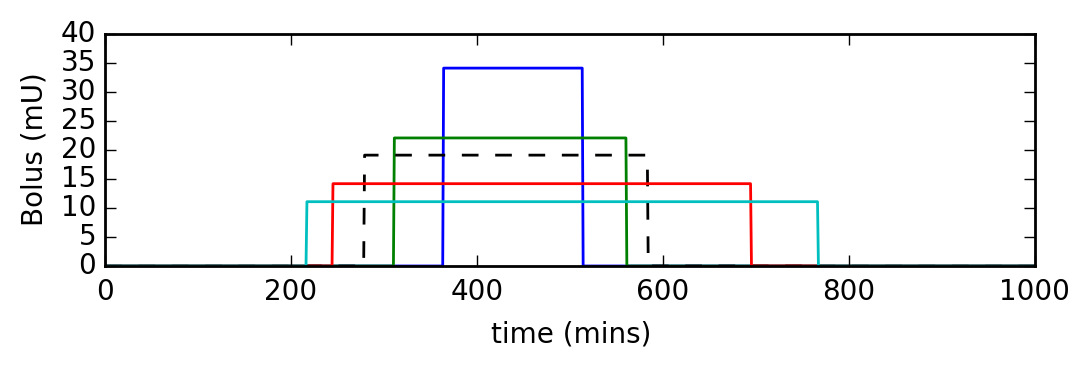}
	\end{flushleft}
	\caption{The functions $u(t,\tau) - \overline u$ for selected input durations for the Hovorka model}
	 \label{fig:hovdurs}
 \end{figure}

 In the example presented in Figures \ref{fig:hovifacresp}--\ref{fig:hovifacinput}, we demonstrate that the results of \cite{town17} apply to the Hovorka model. All values and functions were taken to be as in the previous example with a fixed duration $\tau = 200$. \Cref{fig:hovifacresp} shows three responses of the Hovorka model to a proper pulse delivered at $300, 339$ and $337$, these correspond to the blue, green and red responses, respectively. The green response has two equal maxima bounding the minimum and for which $\gamma = 8.7 \mathrm{mmolL}^{-1} \, (157 \mathrm{mgdl}^{-1})$. The dashed black line is the optimal glucose concentration achieved for this system in the previous example i.e. when both the input time and duration where optimised.

 In \Cref{fig:hovifacinput} the maximum plasma glucose and magnitude of the proper input bolus $\hat u$ is shown as a function of the input time $t'$. The lowest maximum occurs at $t' = 339$, corresponding to the green response in \Cref{fig:hovifacresp}.

\begin{figure}[H]
	 \begin{flushleft}
		 \includegraphics[width=8cm, height=9.7cm, keepaspectratio]{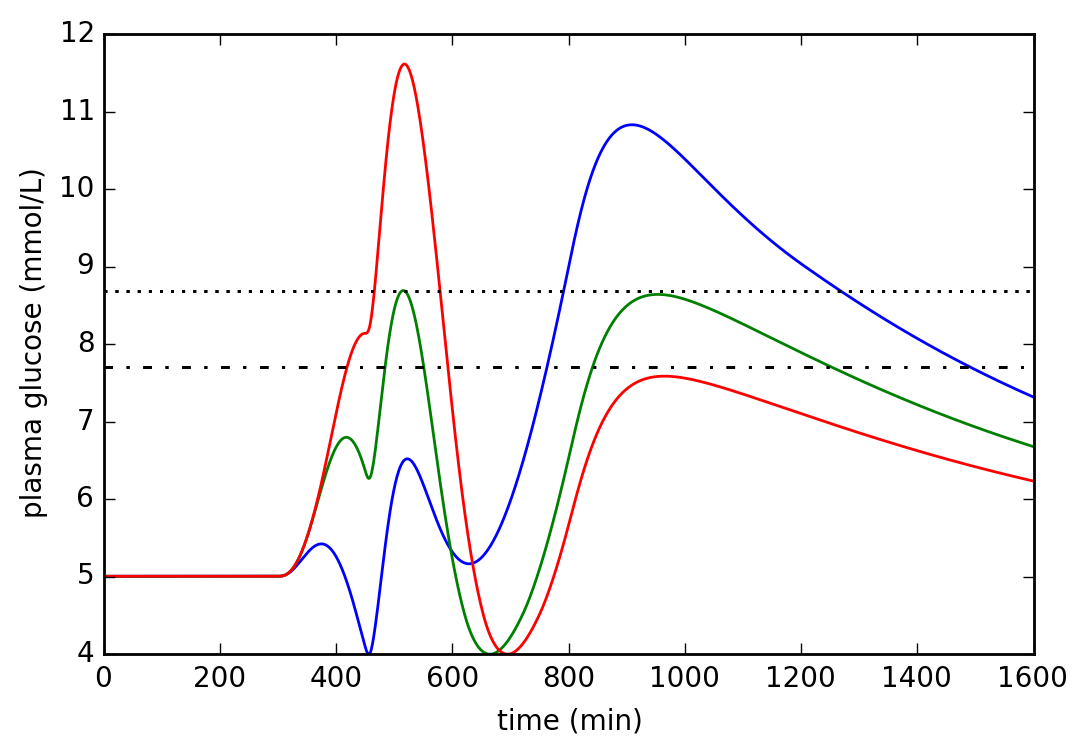}
	\end{flushleft}
	\caption{Glucose responses of the Hovorka model to pulse inputs with fixed duration $\tau =200$ and varying input times.}
	 \label{fig:hovifacresp}
 \end{figure}

\begin{figure}[H]
	 \begin{flushleft}
		 \includegraphics[width=8cm, height=9.7cm, keepaspectratio]{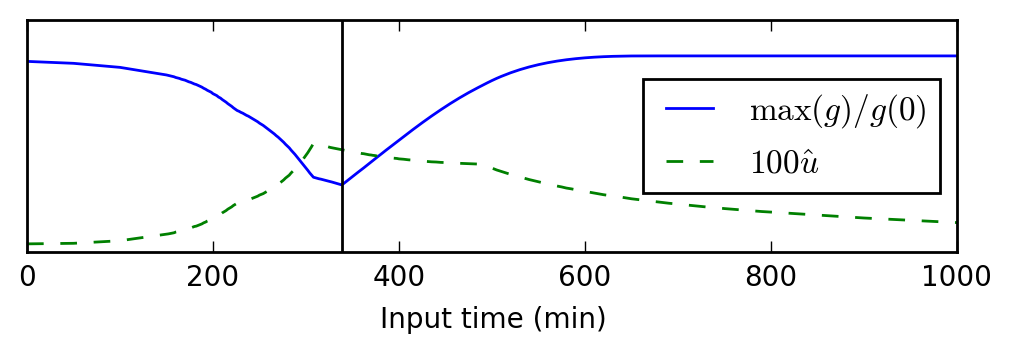}
	\end{flushleft}
	\caption{Normalised maximum plasma glucose and magnitude of the proper input bolus with a fixed duration $\tau=200$ as a function of the input time, for the Hovorka model.}
	 \label{fig:hovifacinput}
 \end{figure}

\section{Conclusions and Further Work}

We have given necessary and sufficient characterisations of the optimality of pulse inputs to the Bergman minimal and Hovorka models in terms of the shape of the predicted plasma glucose concentration. This paper, in conjunction with \cite{town17}, determines the magnitude of the maximum glucose concentration in response to changes in the parameters of a pulse input. These results demonstrate the possibility of rejecting disturbances by tuning the duration and delivery time of a bolus input of some shape.

Current research aims to generalise the presented results to any bounded input function $u(t)$. We are also interested in characterising the behaviour of $\frac{d}{d \tau} \gamma(u(\tau))$ -- the rate of change of the maximum of the response $g$ as a function of the duration $\tau$. This may provide conditions which guarantee the existence of $g(t) > g(0)$, for all durations, or a $\gamma$--optimal input, which are required for \Cref{thm:mainmain}.

Given the general nature of the proofs of the current results we believe it is likely that similar results hold for other models of glucose metabolism.

\bibliographystyle{IEEEtran}
\bibliography{IEEEabrv,references}
\end{document}